\documentclass[a4paper,english, 11pt]{article}
\usepackage{amssymb,amsthm,amsmath,amscd,amsfonts,bbm,mathrsfs}
\usepackage{hyperref}
\usepackage[all,cmtip]{xy}

\theoremstyle{plain}
\newtheorem*{PropA}{Proposition A}
\newtheorem*{PropB}{Proposition B}
\newtheorem*{ThmA}{Theorem A}
\newtheorem{Lemma}{Lemma}
\numberwithin{Lemma}{section}
\numberwithin{equation}{section}

\newtheorem{Proposition}[Lemma]{Proposition}
\newtheorem{Corollary}[Lemma]{Corollary}

\theoremstyle{definition}
\newtheorem{Remark}[Lemma]{Remark}

\newcommand{\N}{\mathbb{N}}
\newcommand{\Z}{\mathbb{Z}}
\newcommand{\Q}{\mathbb{Q}}

\newcommand{\C}{\mathbb{C}}
\newcommand{\F}{\mathbb{F}}

\newcommand{\Pbb}{\mathbb{P}}
\newcommand{\G}{\mathbb{G}}

\DeclareMathOperator{\res}{res}
\DeclareMathOperator{\tr}{tr}

\DeclareMathOperator{\codim}{codim}

\DeclareMathOperator{\Tor}{Tor}

\DeclareMathOperator{\Sym}{Sym}

\DeclareMathOperator{\GL}{GL}
\DeclareMathOperator{\Sp}{Sp}
\DeclareMathOperator{\SL}{SL}
\DeclareMathOperator{\Oo}{O}
\DeclareMathOperator{\SO}{SO}

\DeclareMathOperator{\Spec}{Spec}

\newcounter{listcounter}

\newskip{\itemsepamount}
\newskip{\topsepamount}
\newenvironment{equivlist}{%
  \begin{list}
    {\upshape (\roman{listcounter})}
    {\setlength{\leftmargin}{23pt}
     \setlength{\rightmargin}{0pt}
     \setlength{\itemindent}{0pt}
     \setlength{\labelsep}{5pt}
     \setlength{\labelwidth}{18pt}
     \setlength{\listparindent}{\parindent}
     \setlength{\parsep}{0pt}
     \setlength{\itemsep}{\itemsepamount}
     \setlength{\topsep}{\topsepamount}
     \usecounter{listcounter}}}
  {\end{list}}

\begin{document}

\title{On the Chow Ring of the Classifying Space of some Chevalley Groups}
  \author{Dennis Brokemper}
  \date{November 2016}
  \maketitle
\begin{abstract}
Let $p$ be a prime and $q$ be a power of $p$. We compute the Chow ring of the classifying space of some Chevalley groups 
$G(\F_q)$, when considered as a finite algebraic group over a field of characteristic $p$ containing $\F_q$.
Using specialization from characteristic $0$ to characteristic $p$ we also obtain results over
the complex numbers. 
\end{abstract}

\section*{Introduction}
\addcontentsline{toc}{section}{Introduction}
In this article we are interested in the Chow ring $A^* BG$ of the classifying space $BG$ of a linear algebraic group $G$.
In \cite{To} Totaro shows that this ring coincides with the Chow ring 
of the quotient variety $U/G$ in a suitable range of dimensions, where $U$ is open in a representation $V$ of $G$, such that $G$ acts freely on $U$.
Using this Totaro determines $A^*BG$ for the classical groups. He also treats the case of some finite abstract groups, including the symmetric groups. 

We are especially interested in the case of Chevalley groups, i.e. the finite groups $G(\F_q)$, where $G$ is a connected split 
reductive group scheme over $\Z$. To state our main result we use the following notation
\begin{align*}
 R_1(n,q) & =\Z[c_1,\ldots,c_n]/((q-1)c_1,(q^2-1)c_2,\ldots,(q^n-1)c_n), \\
 R_2(m,q) & =\Z[c_2,c_4,\ldots,c_{2m}]/((q^2-1)c_2,(q^4-1)c_4,\ldots,(q^{2m}-1)c_{2m}), \\
 R_3(n,q) & =\Z[c_2,c_3,\ldots,c_n]/((q^2-1)c_2,(q^3-1)c_3,\ldots,(q^n-1)c_n).
\end{align*}

Moreover, the subscript $\C$ after a finite group means that we consider this group as a finite algebraic group
over $\C$.

\begin{ThmA}
Let $S$ be the product of $p$ and the primes that divide $q-1$. Then the following equations hold.
\begin{align*}
 A^*B(\GL_n(\F_q)_{\C})_{2p} & = R_1(n,q)_{2p} \\
 A^*B(\Sp_{2m}(\F_q)_{\C})_{2p} &= R_2(m,q)_{2p}, \\
 A^*B(\SL_n(\F_q)_{\C})_S &=R_3(n,q)_S,
\end{align*}
where $c_i$ denotes the $i$-th Chern class of the Brauer lift of the canonical representation of the respective groups. \\
If $q \equiv 1$ mod $4$ it suffices to invert $p$ in the first and second equation. 
\end{ThmA}

In \cite{Gu} Guillot computes the mod $l$ Chow ring of the classifying space of $\GL_n(\F_q)_{\C}$ for odd primes $l$ different from $p$. 
If $\F_q$ contains the $l^b$-th roots of unity for some integer $b$ he also computes the mod $l^b$ Chow ring. His results are
$$
A^* B(\GL_n(\F_q)_{\C})/l=\Z/l[c_r,c_{2r},\ldots,c_{mr}],
$$
where $r=[\F_q(\mu_l):\F_q]$ and $m=\lfloor n/r \rfloor$, and
$$
A^* B(\GL_n(\F_q)_{\C})/\l^b=\Z/l^b[c_1,\ldots,c_n]
$$
in the case $\mu_{l^b} \subset \F_q$. This is consistent with our result since $i=r,\ldots,mr$ are precisely the indices where $q^i-1$ is divisible by $l$.

Let us sketch the proof of Theorem A. Instead of computing $A^*BG(\F_q)_{\C}$ directly we first consider $G(\F_q)$ as a finite algebraic group over
$\bar{\F}_p$. As it turns out the computation is much simpler in this case, because
we have another presentation of the stack $BG(\F_q)_{\bar{\F}_p}$. Namely,
it follows from a theorem of Lang-Steinberg that this stack is canonical isomorphic to the quotient stack
$[G_{\bar{\F}_p}/G_{\bar{\F}_p}]$, where the action is given by conjugation with the $q$-th power Frobenius. If $G$ is special, i.e. every $G$-torsor 
is locally trivial for the Zariski toplogy, we can give a complete description of $A^*BG(\F_q)_{\bar{\F}_p}$. This is a consequence of the following Proposition. 

\begin{PropA}
Let $G$ be a connected split reductive group over $\F_q$ with split maximal torus $T$
and let $\sigma \colon G \to G$ be the Frobenius isogeny.
We write $S=\Sym(\hat{T})$, where $\hat{T}$ is the character group of $T$, and $S_+$ for the ideal in $S$ generated by homogeneous elements of positive degree. 
We have a natural action of $\sigma$ on $S$, that we will also denote by $\sigma$. \\
Consider the action of $G$ on itself by $\sigma$-conjugation. 
If $G$ is special we have
$$
 A^*_G(G)=S^{W_G}/(f-\sigma f \mid f \in S_+^{W_G}),
$$
where $W_G=W(G,T)$ denotes the Weyl group of $G$
\end{PropA}

We will prove a more general assertion in Proposition \ref{PropChowConj}.
Applying Proposition A to the special groups $G=\GL_n, \Sp_{2m}$ and $\SL_n$ we obtain
$A^*B\GL_n(\F_q)_k=R_1(n,q)$, $A^*B\Sp_{2m}(\F_q)_k=R_2(m,q)$ and $A^*B\SL_n(\F_q)_k=R_3(n,q)$ for any field $k$ containing $\F_q$. 

We then investigate injectivity and surjectivity of the specialization map
$$
\sigma \colon A^*BG(\F_q)_{\C} \to A^*BG(\F_q)_{\bar{\F}_p}.
$$
to obtain results over $\C$. 

\begin{PropB}
\begin{equivlist}
 \item  For the classical groups $G=\GL_n(\F_q)$, $\Sp_{2m}(\F_q)$, $\Oo_n(\F_q)$ and 
 $\SO_n(\F_q)$ over some finite field $\F_q$ of characteristic $p$ the specialization map $A^*BG_{\C} \to A^*BG_{\bar{\F}_p}$ becomes injective after inverting $2p$.
 \item If $q \equiv 1$ mod $4$ the specialization map for $\GL_n(\F_q)$, $\Sp_{2m}(\F_q)$ and $\Oo_{2m+1}(\F_q)$ becomes injective after inverting $p$.
 \item If $S$ denotes the product of $p$ and all prime divisor of $q-1$, the specialization map for $\SL_n(\F_q)$ becomes injective after inverting $S$.
\end{equivlist}
\end{PropB}

To prove this it suffices to see that the specialization map is
injective for the respective $l$-Sylow subgroups of $G(\F_q)$ by the usual transfer argument.  
These $l$-Sylow subgroups are known to be products of iterated wreath products $(\Z/l)^{\wr i} \wr (\Z/l^b)$ by work of Weir (\cite{We}). 
The Chow ring of the classifying space of this kind of iterated wreath products is computed over the complex numbers by Totaro (\cite{To}), using the cycle map to 
Borel-Moore homology. We carry his proof over to the case of positive characteristic by using etale homology instead of Borel-Moore homology. 
It is an interesting question whether the specialization map is always injective or even an isomorphism after inverting $p$ for an arbitrary finite abstract group. 

For $G=\GL_n, \Sp_{2m}$ and $\SL_n$ the Chow ring $A^*BG(\F_q)_{\bar{\F}_p}$ is generated by Chern classes of representations.
Using the theory of Brauer lifts we see that the specialization map for these groups is surjective. 
\\
\\
{\bf Acknowledgement.} This article is part of the author's PhD thesis. I wish to thank my advisor Eike Lau for his guidance and support
and for many helpful comments on earlier drafts of this article.
\\
\\
\textit{Terminology and Notation.}
All schemes are assumed to be separated and of finite type over the base field $k$. Algebraic groups are affine smooth group schemes over $k$.
The character group of an algebraic group $G$ will be denoted by $\hat{G}$.
Group action always means left group action. 
If $\varphi \colon H \to G$ is a homomorphism of groups and $H$ is a subgroup of $G$ we call the action 
$h \cdot g = h g \varphi(h)^{-1}$ of $H$ on $G$ conjugation with $\varphi$ or $\varphi$-conjugation.
If $G$ is an algebraic group acting on a scheme $X$ we denote by $[X/G]$ the corresponding quotient stack.
If $X = \Spec k$ we sometimes write $BG$ instead of $[\Spec k/G]$ and call $BG$ the classifying space of $G$.
We denoty by $A^G_*(X)$ the $G$-equivariant Chow ring of $X$ defined in \cite{EG2}: For an open subset $U$ of a $G$-representation $V$ such that
$G$ acts freely on $U$ and $\codim(U^c,V)>i$ one defines $A_i^G(X)=A_{i+l-g}((X \times U)/G)$. Here $l$ resp. $g$ is the dimension of $V$ resp. $G$.
We call such a pair $(V,U)$ a good pair for $G$ and $(X \times U)/G$ an approximation of $[X/G]$.
If $X$ is smooth $A_*^G(X)$ carries a ring structure, which 
makes it naturally isomorphic to the equivariant operational Chow ring $A^*_G(X)$ of $X$ (\cite[Proposition 4]{EG2}). 
$A^*_G(X)$ in turn is isomorphic to the operational Chow ring $A^*([X/G])$ of the quotientstack $[X/G]$ provided that $X$ is smooth (\cite[Proposition 19]{EG2}). 
When $X= \Spec k$ is a point we sometimes write $A^*BG$ or $A^*_G$ instead of $A^*_G(\Spec k)$.

\tableofcontents

\section{The equivariant Chow Ring for $\varphi$-Conjugation}
In this section we consider the following situation: Let $G$ be a split reductive group over a field $k$ and let $\varphi \colon L \to M$ be an isogeny
between Levi subgroups $L$ and $M$ of $G$. We are then interested in the equivariant Chow ring $A^*_L(G)$, where $L$ acts on $G$ by $\varphi$-conjugation. 

Before we state our result, we recall the following fact: If $T$ is a split maximal torus of $G$ contained in $L$ there exists a unique element $g_0 \in G(k)$ up to 
multiplication with an element of $N_G(T)$ such that $\varphi(T)={}^{g_0}T(:=g_0Tg_0^{-1})$. This can be seen as follows. Let $K$ denote the kernel of $\varphi$. Then
there is an isomorphism $M \cong L/K$ under which $\varphi$ corresponds to the quotient map $L \to L/K$. Then $\varphi(T)=T/(T \cap K)$ is again
a split maximal torus of $G$ contained in $M$, since $K$ is finite. But two split maximal tori of $G$ are conjugated and the claim follows.

\begin{Proposition}
\label{PropChowConj}
Let $G$ be a connected split reductive group over a field $k$ with split maximal torus $T$. 
Consider an isogeny $\varphi \colon L \to M$, where $L$ and $M$ are Levi components of parabolic subgroups $P$ and $Q$ of $G$.
Assume $T \subset L$ and let $g_0 \in G(k)$ such that $\varphi(T)={}^{g_0}T$.
Let $\tilde{\varphi} \colon T \to T$ denote the isogeny $\varphi$ followed by conjugation with $g_0^{-1}$.
We write $S=\Sym(\hat{T})=A^*_T$ and $S_+=A^{\geq 1}_T$. We have a natural action of $\tilde{\varphi}$ on $S$, 
that we will also denote by $\tilde{\varphi}$. \\
Consider the action of $L$ on $G$ by $\varphi$-conjugation. 
If $G$ is special we have
\begin{equation} \label{EqProp}
 A^*_L(G)=S^{W_L}/(f-\tilde{\varphi} f \mid f \in S_+^{W_G}),
\end{equation}
where $W_G=W(G,T)$ and $W_L=W(L,T)$ denote the respective Weyl groups.
(Note that the action of $\tilde{\varphi}$ on $S^{W_G}$ is independent of the choice of $g_0$ since two choices differ by an element of $N_G(T)$.)
\end{Proposition}

\begin{Remark}
 We will only apply this Proposition in the case $L=G$ and $\varphi$ is the Frobenius isogeny. In this case Proposition \ref{PropChowConj} is the assertion
 of Proposition A in the introduction. However, the proof of this more general assertion is basically the same. 
 If $G$ is not special, equation \eqref{EqProp} still holds for the rational Chow ring, i.e. after tensoring with $\Q$. We leave this aside here
 because we are interested in the Chow ring of the classifying space of finite groups $H$, which are annihilated by the order $|H|$ in positive degree.
\end{Remark}

To prove this we use two lemmas.

\begin{Lemma}
\label{Lefflat}
Let $A \to B$ be a faithfully flat ring homomorphism and $I$ an ideal of $A$. Then $IB \cap A=I$.
\end{Lemma}

\begin{proof}
We have to see that $A/I \to B/IB$ is injective. But this map is the base change of $A \to B$ by $A \to A/I$, thus again faithfully flat and hence injective.
\end{proof}

\begin{Lemma}
 \label{LeRingExt}
\begin{equivlist}
 \item Let $R \subset S$ be an extension of rings. Assume there exists an $R$-linear surjective map $f \colon R^n \to S^n$ for some $n$. Then $R=S$.
\item If $R \subset S \subset T$ is an extension of rings such that $T$ is a free module over $S$ and $R$ of the same finite rank, then $R=S$.
\end{equivlist}
\end{Lemma}

\begin{proof}
(i) The map $f$ induces a surjective $R$-linear map $\wedge^n_R(R^n) \to \wedge^n_S(S^n)$ on the exterior powers. 
In other words $S=Rx$ for an $x \in S^*$. In particular, we find an $r \in R$ such that $x^2=rx$. It follows $x=r \in R$ and hence
$R=S$. Part (ii) follows from (i).
\end{proof}

\begin{proof}(of Proposition \ref{PropChowConj}.)
First note that right translation by $g_0$ induces an isomorphism $G \to G$ of $L$-spaces, where on the source we consider the action of $L$ on $G$ by
$\varphi$-conjugation and on the target the action of $L$ on $G$ by $\tilde{\varphi}$-conjugation. After replacing $M$ resp. $Q$ by ${}^{g_0^{-1}}M$ 
resp. ${}^{g_0^{-1}}Q$ and $\varphi$ by $\tilde{\varphi}$ we may assume that $T \subset M$, $\varphi(T)=T$ and $g_0=1$. \\
To compute $A^*_T(G)$ we consider the action of $T \times T$ on $G$ given by $(g_1,g_2) \cdot g=g_1 g g_2^{-1}$.
Using the embedding $T \hookrightarrow T \times T, g \mapsto (g,\varphi(g))$ we get a morphism $$[G/T] \to [G/(T \times T)]$$ which
is a principal $(T \times T)/T$-bundle. But $(T \times T)/T \cong T$ via the map induced by $(g_1,g_2) \mapsto \varphi(g_1)g_2^{-1}$. 
Hence by \cite[Lemma 2.4]{To2} we have
$$
A^*([G/T]) \cong A^*([G/(T \times T)])/\hat{T}A^*([G/(T \times T)]).
$$
Thus we need to compute $A^*([G/(T \times T)])$. For this let $B$ be a Borel subgroup of $G$ such that $T \subset B \subset P$. We can then identify 
$$A^*([G/(T \times T)])=A^*([G/(B \times T)])=A^*_T(G/B)$$  by the homotopy property of Chow groups. Here $G/B$ denotes the quotient for the usual 
right action of $B$ on $G$. Since $G$ is special we obtain from \cite[Proposition 6.6]{Br} 
$$ 
A^*([G/(T \times T)])\cong S \otimes_{S^{W_G}} S=(S \otimes_{\Z} S)/(1 \otimes f - f \otimes 1 \mid f \in S_+^{W_G}),
$$
where the isomorphism is induced by the two morphisms $S \to A^*([G/(T \times T)])$ given pull-back of the two projections $[G/(T \times T)] \to [\Spec k/T]$.
Now let $\chi \in \hat{T}$ be a character of $T$. Then $\chi$ acts on $A^*([G/(T \times T)]$ by multiplication with the element $\varphi \chi \otimes 1 - 1 \otimes \chi$ 
as follows from the definition of the isomorphism $(T \times T)/T \cong T$. Therefore
$$
A^*_T(G)= S/(f-\varphi f \mid f \in S^{W_G}_+).
$$
From this we deduce the result for $A^*_L(G)$ in the following way. Let us write $I$ for the ideal $(f- \varphi f \mid f \in S^{W_G}_+)$ in $S^{W_G}$. We remark that
$L \cap B$ is a Borel subgroup of $L$ containing $T$. Consider the $L/(L \cap B)$-bundles $[\Spec k/(L\cap B)] \to [\Spec k/L]$ and $[G/(L \cap B)] \to [G/L]$.
Since $L$ is special we obtain from \cite[Proposition 2]{EG} (non-canonical) isomorphisms
\begin{align*}
A^*_L \otimes A^*(L/(L \cap B)) & \cong A^*_{L \cap B}=A^*_T \\
A^*_L(G) \otimes A^*(L/(L \cap B)) & \cong A^*_{L \cap B}(G)=A^*_T(G).
\end{align*}
of $A^*_G$ resp. $A^*_G(G)$-modules.
Since $A^*_L=S^{W_L}$ by \cite[Theorem 1]{EG} and since $A^*(L/(L \cap B))$ is a free abelian group of rank $|W_L|$, we deduce from the first equation that $S$ is a free 
$S^{W_L}$-module of rank $|W_L|$. In particular, $S^{W_L} \hookrightarrow S$ is faithfully flat. It follows $IS \cap S^{W_L}=IS^{W_L}$ and that $S/IS$ 
is a free $S^{W_L}/IS^{W_L}$-module of the same rank $|W_L|$. 
The second equation tells us that $A^*_T(G)$ is a free $A^*_L(G)$-module of rank $|W_L|$. Therefore
$$S^{W_L}/IS^{W_L} \subset A^*_L(G) \subset A^*_T(G)=S/IS$$
and $A^*_T(G)$ is free over $S^{W_L}/IS^{W_L}$ and over $A^*_L(G)$ of the same finite rank $|W_L|$. Hence
$A^*_L(G)=S^{W_L}/IS^{W_L}$ by Lemma \ref{LeRingExt}. 
\end{proof}

\section{The Case of positive Characteristic}
 
\begin{Lemma}
\label{LeLang}
 Let $G$ be a connected algebraic group over an arbitrary field $k$ and $\varphi \colon G \to G$ be an isogeny with only a finite number of fixed points. 
 Consider the action of $G$ on itself by $\varphi$-conjugation and let $S$ denote the stabilizer group scheme of the neutral element. 
 Then there is a canonical 1-isomorphism
 $$ [G/G]\cong BS.$$
\end{Lemma}

\begin{proof}
 It suffices to show that the quotient stack $[G/G]$ is a gerbe that has a section over $k$ whose automorphism group is equal to $S$ (\cite[Lemma 3.21]{LMB}).
Let $T$ be a scheme and $B_1, B_2$ be two principal $G$-bundles over $T$ together with $G$-equivariant maps $B_i \to G$. After replacing $T$
by a suitable fppf covering we may assume that $B_1$ and $B_2$ are trivial. The $G$-equivariant maps $B_i \to G$ then correspond to
$g_i \colon T \to G$ and there exists an isomorphism of principal $G$-bundles respecting the maps $B_i \to G$ if and only if
$g_1$ and $g_2$ lie in the same $G(T)$-orbit. This holds after passing to a suitable covering of $T$ because  the morphism 
$G \to G, g \mapsto g\varphi(g)^{-1}$ is faithfully flat. Indeed the morphism is surjective by \cite[Theorem 10.1]{St} and has finite fibers since $\varphi$ has
only finitely many fixed point. Hence it is also flat by the mirracle flatness theorem (\cite[(21. D) Theorem 51]{Ma}). This shows that $[G/G]$ is a gerbe. 
The section of $[G/G]$ with automorphism group equal to $S$ is given by the trivial $G$-bundle over $\Spec k$ with $G$-equivariant map
$G \to G, g \mapsto g\varphi(g)^{-1}$.
\end{proof}

\begin{Corollary}
 \label{CorLang}
 Let $G$ be a connected algebraic group over $\F_q$. Let $G$ act on itself by conjugation with the $q$-th power Frobenius. 
 Then there is a canonical $1$-isomorphism
 $$
 B(G(\F_q)_{\F_q}) \cong [G/G].
 $$
\end{Corollary}

\begin{proof}
This follows from the previous lemma applied to the case that $\varphi$ is the $q$-th power Frobenius.
\end{proof}

\begin{Corollary}
\label{Cor1}
Let $k$ be a field containing $\F_q$. Then the following equations hold.
\begin{align*}
 A^*B(\GL_n(\F_q)_k) &= \Z[c_1,\ldots,c_n]/((q-1)c_1,\ldots,(q^n-1)c_n) \\
 A^*B(\Sp_{2m}(\F_q)_k) &=\Z[c_2,c_4,\ldots,c_{2m}]/((q^2-1)c_2,(q^4-1)c_4,\ldots,(q^{2m}-1)c_{2m}), \\
 A^*B(\SL_n(\F_q)_k) &=\Z[c_2,c_3,\ldots,c_n]/((q^2-1)c_2,(q^3-1)c_3,\ldots,(q^n-1)c_n),
\end{align*}
where the $c_i$ are the $i$-th Chern classes of the canonical representation of the respective groups.  
\end{Corollary}

\begin{proof}
Let $G$ be one of the groups $\GL_n$, $\Sp_{2m}$ or $\SL_n$. By Corollary \ref{CorLang} we have $B(G(\F_q)_k)=[G_k/G_k]$, 
where the action is given by conjugation with the $q$-th power Frobenius.
Moreover, $G$ is special (\cite[Section 4.4]{Se}) and therefore the theorem follows from Proposition \ref{PropChowConj} applied to the 
case $L=G$ and the Frobenius isogeny, using that $S^{W_G}=A^*BG$ is computed in \cite[Section 15]{To}.
\end{proof}

\section{Specialization}
\textit{Specialization for Chow Rings.}
Let $R$ be a discrete valuation ring with fraction field $K$ and residue field $k$ and $i \colon \Spec k \hookrightarrow \Spec R$. 
Let $X$ be of finite type and separated over $S=\Spec R$. 
Then there exists a unique map $\sigma_R \colon A_* X_K \to A_* X_k,$ such that the diagram
$$
\xymatrix{
A_*(X/S) \ar[r] \ar[d]_{i^!} & A_*(X_K) \ar[dl]^{\sigma_R} \\
A_*(X_k)
}
$$
commutes. Here $i^!$ is the refined Gysin homomorphism of $i$ and $A_*(X/S) \to A_*(X_K)$ is the pull-back along the open immersion $X_K \hookrightarrow X$. 
The map $\sigma_R$ is called specialization map. For smooth $X$ $\sigma_R$ is a morphism of rings (\cite[Section 20.3]{Fu}).
On the level of cycles $\sigma_R$ is defined by the map 
 $$
 Z_*(X_K) \to Z_*(X_k), \quad [V] \mapsto [\bar{V}_k],
 $$
where $\bar{V}$ denotes the closure of $V$ in $X$. 

Equivariant chow groups are more generally defined for actions of affine smooth group schemes $G$ defined over $S= \Spec R$ 
with $R$ a discrete valuation ring (or more generally a Dedekind domain). Namely, by \cite[Lemma 7]{EG2} we can find for any $i>0$ a finitely generated projective 
$S$-module $E$ such that $G$ acts freely on an open subset $U$ of $E$ whose complement has codimension greater than $i$. 
For such an $U$ one defines $$A^G_i(X)=A_{i+l-g}(X \times^G U)$$ where $l= \dim(U/S)$ and $g=\dim(G/S)$. All the functorial properties of 
equivariant intersection theory remain valid. In particular, if $X$ is smooth over $S$ there is an intersection product on $A^G_*(X)$.  

For every pair $(E,U)$ as above we get a homomorphism of graded rings
$$ A^*(U_K/G_K) \to A^*(U_k/G_k).$$
These morphisms induce a map
$$
\sigma_R \colon A^*_{G_K} \to A^*_{G_k}
$$
of graded rings. In particular, there exist such a map for any finite abstract group $G$ viewed as a finite algebraic group over $S$.
We have the following naive criterion for the specialization map to be an isomorphism.

\begin{Proposition}
 \label{PropSpez}
Let $G$ be a finite abstract group viewed as a finite algebraic group over a discret valuation ring with quotient field $K$ and residue field $k$ of mixed characteristic.
Then the following assertions hold:
\begin{equivlist}
 \item The specialization map is surjective if $A^*_{G_k}$ is generated as a $\mathbb{Z}$-algebra by Chern classes of representations of $G_k$.
 \item The specialization map is injective if this holds true for every Sylow subgroup of $G$. 
\end{equivlist}
\end{Proposition}

\begin{proof}
 To prove (i) we use the theory of Brauer lifts (\cite[Chapter 18]{Se2}).
Let $R_K(G)$ resp. $R_k(G)$ denote the representation ring of $G_K$ resp. $G_k$. Consider the diagram
$$
\xymatrix{
R_K(G) \ar[r]^d \ar[d]^{c_i} & R_k(G) \ar[d]_{c_i} \\
A^i_{G_K} \ar[r]^{\sigma} & A^i_{G_k}. 
}
$$
Here the maps $c_i$ are induced by the $i$-th Chern class (see Lemma \ref{LemChernclassVR} below) and $d$ is defined as follows. 
If $V$ is a $K$-representation of $G$ we choose a $G$-invariant $R$-lattice $\tilde{V}$ in $V$. The class $[V]$ is then mapped under $d$ to the class 
$[\tilde{V} \otimes k]$ in $R_k(G)$. This class does not depend on the choice of a lattice (\cite[Section 15.2]{Se2}).
By definition of the specialization map and \cite[Proposition 6.3]{Fu} we see that the above diagram is commutative. 
The map $d$ is surjective by \cite[Section 16.1]{Se2}.
In other words we can lift Chern classes in $A^*_{G_k}$ to Chern classes in $A^*_{G_K}$. Hence the specialization map is surjective if $A^*_{G_k}$ is generated 
by Chern classes of representations of $G_k$. 

If $P$ is an $l$-Sylow subgroup of $G$ for some prime $l$ we obtain a diagram
$$
\xymatrix{
(A^*_{G_K})_{(l)} \ar@{^{(}->}[r]^{\res^{G_K}_{P_K}} \ar[d]_{\sigma} & (A^*_{P_K})_{(l)} \ar[d]_{\sigma} \\
(A^*_{G_k})_{(l)} \ar@{^{(}->}[r]^{\res^{G_k}_{P_k}} & (A^*_{P_k})_{(l)}.
}
$$
This diagram is commutative by \cite[Proposition 6.2 (b)]{Fu}. The injectivity of the horizontal maps follows from the fact, that
the composition of restriction map followed by transfer map $A^*_G \stackrel{\res^G_P}{\rightarrow} A^*_P \stackrel{\tr^G_P}{\rightarrow} A^*_G$ is multiplication 
with the index $[G:P]$, which is prime to $l$ (\cite[Section 2.5]{To2}). Hence (ii) follows.
\end{proof}

\begin{Lemma}
 \label{LemChernclassVR}
If $G$ is a finite abstract group and $k$ an arbitrary field, then for each $i \in \Z_{\geq 0}$ there are unique maps $c_i \colon R_k(G) \to A^i_{G_k}$ satisfying 
the following two properties:
\begin{equivlist}
 \item For any $G_k$-representation $V$ one has $c_i([V])=c_i(V) \in A^i_{G_k}$.
 \item $c_i(E_1+E_2)=\sum_{k+l=i} c_k(E_1)c_l(E_2)$ for $E_1,E_2 \in R_k(G)$.
\end{equivlist}
\end{Lemma}

\begin{proof}
 Each virtual representation $E$ of $G$ has a unique expression $E= \sum_{L \in S(G)} \lambda_i [L]$ with $\lambda_i \in \Z$. Here $S(G)$ 
 denotes the set of isomorphism classes of simple representations of $G$. The properties $(i)$ and $(ii)$ then determine $c_i(E)$ uniquely.
\end{proof}

\begin{Remark}
Condition (i) of the above proposition does not hold in general.
An example is given by the symmetric groups (\cite[Section 4]{To}).
\end{Remark}

Let $K$ be a finite field extension of $\mathbb{Q}_p$ and $R$ the integral closure of $\mathbb{Z}_p$ in $K$. Since $\mathbb{Q}_p$ is complete for the $p$-adic 
valuation $R$ is again a discrete valuation ring. Also its residue field $k$ is a finite field extension of $\mathbb{F}_p$. 
If $K'$ is another finite extension of $\Q_p$ containing $K$ with integral closure $R'$ of $\Z_p$ in $K'$ and residue field $k'$ then it is easy to see
that the following diagram commutes
$$
\xymatrix{
A^*_{G_{K'}} \ar[r]^{\sigma_{R'}}  & A^*_{G_{k'}}  \\
A^*_{G_K} \ar[r]^{\sigma_R} \ar[u] & A^*_{G_k} \ar[u]
}
$$
Here the vertical maps are induced by pull-back. 
We can thus take the direct limit over all finite field extensions $K$ of $\Q_p$ of the specialization maps $\sigma_R$ and obtain a map
$$ A^*_{G_{\bar{\Q}_p}} \to  A^*_{G_{\bar{\F}_p}} $$
for every finite group $G$. Fixing an isomorphism $\mathbb{C} \cong \bar{\mathbb{Q}}_p$ we obtain a homomorphism of graded rings 
$$
\sigma_G \colon A^*_{G_{\mathbb{C}}} \to A^*_{G_{\bar{\mathbb{F}}_p}}.
$$ 

\textit{Specialization for etale Homology.}
For any scheme $S$ over a separably closed field $k$ etale 
homology is defined as $$ H_i(S,\Z_l)=H^{-i}R\Gamma(S,T_S) $$ for $l \neq $ char $k$, where $T_S$ is the dualizing complex of $S$. If $a$ denotes the structure map of $S$ 
then $T_S=R a^!\Z_l \in D(S,\Z_l)$. Let us recall the properties of etale homology we shall need. For the proof we refer to \cite{La}.

\begin{Proposition}
\label{PropEtaleHomology}
 Let $k$ be a separably closed field and $l$ be a prime different from the characteristic of $k$. Let $X$ be a scheme over $k$ of dimension $d$.
\begin{equivlist}
 \item $H_i(X,\Z_l)=0$ for $i<0$ and $i>2d$ and $H_{2d}(X,\Z_l)$ is freely generated by the irreducible components of $X$ of dimension $d$.
 \item If $X$ is smooth then $H_i(X,\Z_l)=H^{2d-i}(X,\Z_l)$, where the right hand side denotes the usual $l$-adic cohomology groups.
 \item (Functoriality) Let $f \colon X \to Y$ be proper resp.\ flat of relative dimension $n$. Then $f$ induces an additive push-forward map 
 $f_* \colon H_*(X,\Z_l) \to H_*(Y,\Z_l)$ resp.\ pull-back map $f^* \colon H_*(Y,\Z_l) \to H_{* + n}(X,\Z_l)$ compatible with composition. 
\item  If $f \colon X \to Y$ is finite and locally free of degree $n$ then the composition 
$$
\xymatrix{
H_*(Y,\Z_l) \ar[r]^{f^*} & H_*(X,\Z_l) \ar[r]^{f_*} & H_*(Y,\Z_l) 
}
$$
is multiplication by $n$.
\item(Künneth Formula) If $Y$ is another scheme over $k$ there is an exact sequence of the form
$$
\xymatrix{
0 \ar[r] & \bigoplus_{r+s=i} H_r(X) \otimes H_s(Y) \ar[r] & H_i(X \times Y) &&&
}
$$
$$
\xymatrix{
&&&& \ar[r] & \bigoplus_{r+s=i-1} \Tor_1(H_r(X),H_s(Y)) \ar[r] & 0.
}
$$
\item(Cycle Map) There is an additive cycle map $cl_X \colon A_*(X) \to H_{2*}(X,\Z_l)$ compatible with proper push-forward, flat pull-back and Chern classes.
If $X$ is smooth then $cl_X$ defines a morphism of rings.
\end{equivlist}
\end{Proposition}

Since we do not have a reference for the following lemma, we prove it here.

\begin{Lemma}
\label{LeLES}
 Let $Z \subset S$ be a closed subscheme and denote by $U$ its complement in $S$. Denote the inclusion $Z \hookrightarrow S$ resp.\ $U \hookrightarrow S$ 
 by $i$ resp.\ $j$. Then there is a long exact sequence
 $$
 \xymatrix{
 \ldots \ar[r] & H_{i+1}(U,\Z_l) \ar[r] & H_i(Z,\Z_l) \ar[r]^{i_*} & H_i(S,\Z_l) \ar[r]^{j^*} & H_i(U,\Z_l) \ar[r] &
 }
 $$
\end{Lemma}

\begin{proof}
There is an exact triangle in $D(S,\Z_l)$ of the form
$$
\xymatrix{
i_*Ri^! T_S \ar[r] & T_S \ar[r] & (Rj_*)j^*T_S \ar[r] & i_*Ri^! T_S[1]
}
$$
(See \cite[p. 123]{KW}). Applying $R\Gamma(S,\cdot)$ then gives an exact triangle
$$
\xymatrix{
R\Gamma(Z,T_Z) \ar[r] &  R\Gamma(S,T_S) \ar[r] & R\Gamma(U,T_U) \ar[r] & R\Gamma(Z,T_Z)[1]
}
$$
in $D(\Spec k, \Z_l)$. Taking homology then yields the desired long exact sequence.
\end{proof}

Let $G$ be an algebraic group over a separably closed field $k$ and $l$ be a prime different from the characteristic of $k$. 
The $i$-th $l$-adic cohomology group of $BG$ is defined analogously to the case of Chow groups:
Let $(V,U)$ be a good pair for $G$ with $\codim(U^c) > (i+1)/2$ then
$$
H^i(BG,\Z_l)=H^i(U/G,\Z_l).
$$
By using Proposition \ref{PropEtaleHomology} (ii) and Lemma \ref{LeLES} above one shows in the same way as in the case of Chow groups 
(\cite[Definition-Proposition 1]{EG2}), 
that the above definition is independent of the choice of the pair $(V,U)$ provided that $\codim U^c > (i+1)/2$.

\begin{Lemma}
\label{LeCyclemapmu}
 Let $k$ be a field of arbitrary characteristic and $l$ a prime different from char $k$. Assume $(n,\text{char } k)=1$ then the cycle map
$$
A^*_{\mu_{n,k}} \otimes \Z_l \to H^*(B\mu_{n,k^{sep}},\mathbb{Z}_l)
$$
is an isomorphism. 
\end{Lemma}

\begin{proof}
We may assume $k=k^{sep}$. The natural map 
$$X=(\mathbb{A}_k^{r+1}-\{0\})/\mu_n \to (\mathbb{A}_k^{r+1}-\{0\})/\mathbb{G}_m=\mathbb{P}_k^r$$
is a principal $\mathbb{G}_m$-bundle and the corresponding 
line bundle is given by $\mathcal{O}_{\mathbb{P}^r}(n)$. In other words $B \mu_n$ can be approximated by the complement $X$ of the zero section in 
$\mathcal{O}_{\mathbb{P}^r}(n)$. We will show that the cycle map $A^i(X) \otimes \Z_l \to H^{2i}(X,\Z_l)$
is an isomorphism and $H^i(X,\Z_l)=0$ for odd $i$. Consider the diagram
$$
\xymatrix@C=1em{
A^i(\mathbb{P}^r_k) \otimes \Z_l \ar[r]^-{f_i} \ar[d] & A^{i+1}(\mathcal{O}_{\mathbb{P}^r}(n)) \otimes \Z_l \ar[r] \ar[d] & A^{i+1}(X) \otimes \Z_l \ar[d] \ar[r] & 0  \\
H^{2i}(\mathbb{P}^r_k,\Z_l) \ar[r]^-{g_i} & H^{2i+2}(\mathcal{O}_{\mathbb{P}^r}(n),\Z_l) \ar[r] & H^{2i+2}(X,\Z_l) \ar[r] & H^{2i+1}(\mathbb{P}^r_k,\Z_l)
}
$$
with exact rows. Here the lower row comes from Lemma \ref{LeLES}. This diagram is commutative since the cycle map is compatible with proper push-forward 
and flat pull-back by Proposition \ref{PropEtaleHomology} (vi). It is well known that the first vertical map is an isomorphism. Hence the second vertical map is 
also an isomorphism. Since $H^{2i+1}(\mathbb{P}^r_k,\Z_l)=0$ the first claim follows. 
For the second claim it suffices to see that the $g_i$ or equivalently $f_i$ is injective.
But the composition
$A^i(\mathbb{P}^r_k) \to A^{i+1}(\mathcal{O}_{\mathbb{P}^r}(n)) \xrightarrow{\cong} A^{i+1}(\mathbb{P}^r_k)$ is capping with $c_1(\mathcal{O}_{\mathbb{P}^r}(n))$ 
by the self intersection formula, and under the identification $A^i(\mathbb{P}^r_k)=\Z=A^{i+1}(\mathbb{P}^r_k)$ this corresponds 
to multiplication with $n$.
\end{proof}

Let $R$ be a discrete valuation ring with fraction field $K$ of characteristic $0$ and perfect residue field $k$ of characteristic $p$
and let $X \to \Spec R$ be smooth. By \cite[Expose 10, 7.13-7.16]{SGA6} there exists an etale specialization map
$$
\sigma_{R} \colon H^{i}(X_{\bar{K}},\Z_l) \to H^{i}(X_{\bar{k}},\Z_l)
$$
for $l$ a prime different from $p$ that is compatible with the specialization map for the Chow ring under the cycle map, i.e. the diagram
$$
\xymatrix{
H^{2i}(X_{\bar{K}},\Z_l) \ar[r]^{\sigma_R} & H^{2i}(X_{\bar{k}},\Z_l) \\
A^i(X_K) \ar[u] \ar[r]_{\sigma_R} & A^i(X_k) \ar[u]
}
$$
commutes.

\begin{Lemma}
\label{LeSpezTor}
 Let $G$ be a finite abstract group and $X \to Y$ be a $G$-torsor over $R$. If $\sigma_R \colon H^i(X_{\bar{K}},\Z/l^n) \cong H^i(X_{\bar{k}},\Z/l^n)$ is an isomorphism
 for all $i$, the same is true for $\sigma_R \colon H^i(Y_{\bar{K}},\Z/l^n)\to H^i(Y_{\bar{k}},\Z/l^n)$.
\end{Lemma}

\begin{proof}
 Consider the Hochschild-Serre spectral sequences
 \begin{align*}
  & H^p(G,H^q(X_{\bar{K}},\Z/l^n)) \Rightarrow H^{p+q}(Y_{\bar{K}},\Z/l^n) \\
  & H^p(G,H^q(X_{\bar{k}},\Z/l^n)) \Rightarrow H^{p+q}(Y_{\bar{k}},\Z/l^n). 
 \end{align*}
 Since specialization is compatible with pull-back, the map $\sigma_R \colon H^i(X_{\bar{K}},\Z/l^n)\to H^i(X_{\bar{k}},\Z/l^n)$ is an isomorphism
 of $G$-modules, and it extends to an isomorphism of the spectral sequences compatible with the specialization map for $Y$. The lemma follows. 
\end{proof}

\begin{Corollary}
 \label{CorSpezCoh}
 Let $G$ be a finite abstract group. Then the specialization map $H^*(BG_{\bar{K}},\Z_l) \to H^*(BG_{\bar{k}},\Z_l)$ is an isomorphism
\end{Corollary}

\begin{proof}
Choose a good pair $(E,U)$ for $G_R$. Then $\sigma_R \colon H^i(U_{\bar{K}},\Z_l) \to H^i(U_{\bar{k}},\Z_l)$ is an isomorphism for all $i < 2 \codim(E-U)-1$
by Lemma \ref{LeLES}. Since we can choose this codimension to be arbitrary high the assertion follows from the previous Lemma.
\end{proof}

\begin{Remark}
 It follows in fact from the Hochschild-Serre spectral sequence that $H^i(BG_{k^{sep}},\Z_l)=H^i(G,\Z_l)$, where $G$ is a finite abstract group and
 $l$ is a prime different from the characteristic of $k$. For this note that if $U$ is open in a representation $V$ of $G$ then
 $H^q(U,\Z/l^n)=0$ for $0 < q < 2\codim(U^c,V)-1$ and $H^0(U,\Z/l^n)=\Z/l^n$.
\end{Remark}

\section{Specialization for Wreath Products}

We say that a scheme admitts a cell decomposition if it can be stratified into a finite disjoint union of open subsets of affine spaces.
Our goal is to prove the following proposition which is inspired by \cite[Lemma 8.1]{To}, a variant of which is stated in Lemma \ref{LeSplitInjective} below.

\begin{Proposition}
\label{PropWP}
Let $p \neq l$ be prime numbers. Assume $G$ is a finite abstract group satisfying the following conditions:
\begin{equivlist}
 \item The specialization map $A^*_{G_{\bar{\Q}_p}} \otimes \Z_l \to A^*_{G_{\bar{\F}_p}} \otimes \Z_l$ is an isomorphism.
 \item $BG_{\bar{\Q}_p}$ and $BG_ {\bar{\F}_p}$ can be approximated by schemes admitting a cell decomposition.
 \item The cycle map $A^*_{G_{\bar{\Q}_p}} \otimes \Z_l \to H^*(BG_{\bar{\Q}_p},\Z_l)$ is split injective. 
\end{equivlist}
Then the same conditions hold for the wreath product $\Z/l \wr G$.
\end{Proposition}

In order to prove this proposition we need to say something about the cyclic product of a quasi-projective scheme, since these are the 
spaces that approximate $B(\Z/l \wr G)$. 

\textit{Cyclic Products.} Let $S$ be a quasi-projective scheme over an arbitrary field $k$. Let $l$ be a prime. Consider the permutation action of $\Z/l$ on $S^l$.
Since $S$ is quasi-projective the geometric quotient $S^l/(\Z/l)$ exists.
If we take out the diagonal of $S^l$, then the action of $\Z/l$ on $S^l-S$ is free and 
$$\pi \colon S^l-S \to (S^l-S)/(\Z/l)$$ 
is a principal bundle quotient with structure group $\Z/l$. We will write $X=S^l-S$, $Y=(S^l-S)/(\Z/l)$ and $Z^lS=S^l/(\Z/l)$ and call $Z^lS$ the cyclic product of $S$. 
Note that $\pi$ is finite, etale of degree $l$. 

Assume $(l,\text{char }k)=1$ and $k$ contains the $l$-th roots of unity. Fix a character $\mu_l \hookrightarrow \mathbb{G}_m$ and let $c_1 \in A^1 Y$ 
be its first Chern class. For $i>0$ and $j\leq il-1$ we consider the operations 
$$
\gamma_i \colon A_i(S) \to A_{li}(Y), \quad \alpha_i^j \colon A_i(S) \to A_j(Y)
$$
constructed in \cite[Section 7]{To}, where $\alpha_i^j=c_1^{li-j} \gamma_i$. We have 
$$\pi^* \gamma_i(a)=a^{\otimes l}|_X, \text{ hence } l \gamma_i(a)=\pi_*(a^{\otimes l}|_X),$$
and $\alpha_i^j$ is a homomorphism of abelian groups. 

Let us recall the construction of $\gamma_i$. 
Let $C\in Z_*(S)$ be a cycle on $S$ of dimension greater than $0$. Then the support of $C^l$ is not contained in the diagonal of $S^l$ and we may consider the
restriction $C^l-C$ of $C^l$ to a cycle on $X=S^l-S$. The cycle $C^l-C$ is invariant under the action of $\Z/l$ and hence is the pull-back of a unique cycle $Z^l(C)-C$
on $Y$ under the etale map $X \to Y$. This defines a map $Z_{\geq 1}(S) \to Z_{\geq l}(Y)$ which passes through rational equivalence (see loc. cit.) and induces the maps
$\gamma_i$. Note that $\gamma_i$ is not additive. More precisely, let $x,y \in A_i(X)$ then by \cite[Lemma 7.1 (3)]{To} we have
\begin{equation} \label{EqSum}
\gamma_i(x+y)=\gamma_i(x) + \sum \pi_*(\alpha_1 \otimes \ldots \otimes  \alpha_l) + \gamma_i(y),
\end{equation}
where the sum is over a set of representatives $(\alpha_1,\ldots,\alpha_l)$ of the $\Z/l$-orbits in $\{x,y\}^l - \{(x,\ldots,x),(y,\ldots,y)\}$.

\begin{Remark}
 In fact, Totaro's operations map into the Chow group of $Z^lS$, so our operations are Totaro's composed with the pull-back to the open subset $Y$ in $Z^lS$.
However, in the end we will only be interested in the Chow group resp.\ homology of $Y$ for dimension greater than $\dim S$ resp.\ $2 \dim S$ and in this range
the Chow group resp.\ homology of $Z^lS$ and $Y$ coincide. 
\end{Remark}

Totaro then defines a functor $F_l$ from graded abelian groups to graded abelian groups in the following way.
Let $A_*$ be a graded abelian group. Then $F_l(A_*)$ is the graded abelian group generated by the graded abelian group $A_*^{\otimes l}$ 
together with $A_i \otimes Z/l$ in degree $j$ for $i+1\leq j \leq li-1$ and elements $\gamma_i x_i$ in degree $li$ for $x \in A_i$ and $i>0$ subject to the relations

\begin{align*}
 x_1 \otimes \ldots \otimes x_l &= x_2 \otimes \ldots \otimes x_l \otimes x_1 \\
 l\gamma_ix &= x^{\otimes l} \\
 \gamma_i(x+y) &= \gamma_i x + \sum_{\alpha} \alpha_1 \otimes \ldots \otimes \alpha_l + \gamma_i(y).
\end{align*}
Here $\alpha$ runs through the $\Z/l$-orbits in $\{x,y\}^l - \{(x,\ldots,x),(y,\ldots,y)\}$. 
If $A_*$ is isomorphic to a finite direct sum $\bigoplus_{i=1}^n \Z/a_i \cdot e_i$, where each $e_i$ is homogeneous and $a_i$ is $0$ or a prime power,
we can give a more precise description of $F_l(A_*)$ in the following way. Let $R$ be the set of $i$ such that $\deg e_i>0$ and $a_i=0$ or a power of $l$, then 
$$
F_l(A_*)=\bigoplus_{\underline{i} \in (\{1,\ldots,n\}^l-R)/(\Z/l)}  \Z/(a_1,\ldots,a_l)  \cdot e_{i_1}\otimes \ldots \otimes e_{i_l} \oplus
\bigoplus_{i\in R} \Z/(la_i) \cdot \gamma(e_i)
$$
$$
 \oplus \bigoplus_{\stackrel{i \in R}{\dim e_i < j < l\dim e_i}} \Z/l \cdot \alpha^j(e_i) 
$$
where again $\{1,\ldots,n\}^l-R$ denotes the complement of $R$ embedded diagonally in $\{1,\ldots,n\}^l$. 
Using the operations 
\begin{align*}
 \otimes l &\colon A_i(S) \to A_{li}(X) \xrightarrow{\pi_*} A_{li}(Y) \\
 \gamma_i &\colon A_i(S) \to A_{il}(Y) \\
 \alpha_i^j &\colon A_i(S) \to A_j(Y)  
\end{align*}
we obtain a homomorphism of graded abelian groups
$$
\Psi_k \colon F_l(A_*(S)) \to A_*(Y).
$$
We shall need one last piece of notation.
For a scheme $S$ and $r \in \N$ we write $F^{< r}_l(S)$ for the subgroup of elements of $F_l(A_*S)$ of degree $> l \dim S -r$.
Clearly $F^{< r}_l(S)=F_l(A^{< r}S)$. 

\begin{Lemma}
 \label{LeCommutative}
 Let $S$ be a smooth quasi-projective scheme over $\Z_p$. Then the diagram
 $$
 \xymatrix{
 F_l^{< r}(S_{\bar{\Q}_p}) \ar[r]^{\Psi_{\bar{\Q}_p}} \ar[d] & A^{< r}(Y_{\bar{\Q}_p}) \ar[d]  \\
 F_l^{< r}(S_{\bar{\F}_p}) \ar[r]^{\Psi_{\bar{\F}_p}} & A^{< r}(Y_{\bar{\F}_p}) 
 }
 $$
 commutes. Here the vertical maps are given by specialization.
\end{Lemma}

\begin{proof}
 First we note that the exterior product map is compatible with specialization meaning that the diagram
 $$
 \xymatrix{
 A_*(S_K \times_K S_K) \ar[r]^{\sigma_R} & A_*(S_k \otimes_k S_k) \\
 A_*(S_K) \otimes A_*(S_K) \ar[u]^{\times} \ar[r]_{\sigma_R^{\otimes 2}}  & A_*(S_k) \otimes A_*(S_k) \ar[u]_{\times}
 }
 $$
 commutes. This follows easily from the definition of the specialization map on the level of cycles. Moreover, push-forward and intersecting with 
 Chern classes of line bundles are compatible with pull-back and refined Gysin homomorphisms. Hence from the definition of the maps $\Psi_{\bar{\Q}_p}$ and         
 $\Psi_{\bar{\F}_p}$ we see that it suffices to show that the operation $\gamma$ is compatible with specialization. 
 
 Using equation \eqref{EqSum} we see it suffices to prove this for a fundamental cycle $x=[V] \in A_i(S_{\bar{\Q}_p})$.
 For this consider a finite extension $K$ of $\Q_p$ such that $V$ is defined over $K$.
 Let $A$ be the integral closure of $\Z_p$ in $K$ and $k$ the residue field of $A$. 
 Replace $S$ by $S_A$ and write $Y=(S^l-S)/(\Z/l)$. We have to see $\gamma_{k}(\sigma_A([V]))=\sigma_A(\gamma_{K} [V])$.
 Let $\bar{V}$ be the closure of $V$ in $S$, then $\sigma_A([V])=[\bar{V}_k]$. It follows $\gamma_{k}(\sigma_A([V]))=[(\bar{V}_k^l-\bar{V}_k)/(\Z/l)]$.
 Now since $V$ is geometrically integral $(\bar{V}^l-\bar{V})/(\Z/l)$ is again a subvariety of $Y$ mapping dominantly to $\Spec A$
 and resitricting to $(V^l-V)/(\Z/l)$. It follows $\sigma_A(\gamma_{K} [V])=[(\bar{V}_k^l-\bar{V}_k)/(\Z/l)]$ as wanted.
\end{proof}

\begin{Lemma}
 \label{LeSplitInjective}
 Let $S$ be a smooth quasi-projective scheme over $\bar{\Q}_p$ and $r \leq \dim S$. Assume that $A^{< r}(S) \otimes \Z_l \to H^{<2r}(S,\Z_l)$ is split injective.
 Then the composition
 $$
 F_l^{<r}(S) \otimes \Z_l \to A^{<r}(Y) \otimes \Z_l \to H^{<2r}(Y,\Z_l)
 $$
 is split injective.
\end{Lemma}

This Lemma is a variant of \cite[Lemma 8.1 (3)]{To}:

\begin{Lemma}(Totaro)
\label{LeTo}
Let $S$ be a quasi-projective scheme over $\C$. If the cycle map $A_*(S) \to H_*^{BM}(S,\Z)$ is split injective, then the composition
$F_l(S) \to A_*(Z^lS) \to H_*^{BM}(Z^lS,\Z)$ is split injective.
\end{Lemma}

Here $ H_*^{BM}(S,\Z)$ denotes Borel-Moore homology (\cite[Section 19.1]{Fu}). Let us explain how the two lemmas above are related.
Choose an isomorphism $\bar{\Q}_p \cong \mathbb{C}$. By the comparison theorem (\cite[XI]{SGA4}) we have 
$$
H^{<2r}(Y,\Z_l) = H_{sing}^{<2r}(Y,\Z) \otimes \Z_l
                = H^{\text{BM}}_{>2l\dim S-2r}(Y,\Z) \otimes \Z_l.
$$
Note that $2l\dim S -2r \geq 2 \dim S$ and therefore 
$$H^{\text{BM}}_{>2l\dim S-2r}(Y,\Z)=H^{\text{BM}}_{>2l\dim S-2r}(Z^lS,\Z).$$
Under this identifications Lemma \ref{LeSplitInjective} follows from Lemma \ref{LeTo} and its proof, using that elements of a certain degree
are not affected by elements of larger degrees.

\begin{proof}(of Proposition \ref{PropWP})
 Fix $r_o \in \mathbb{N}$. Let $S \to \Z_p$ be a smooth approximation of $BG_{\Z_p}$ up to codimension $r_o$.
 We then have $A^r(S_{\bar{\Q}_p})=A^r_{G_{\bar{\Q}_p}}$ as well as 
 $A^r_{\Z/l \wr G_{\bar{\Q}_p}}=A^r((S^l_{\bar{\Q}_p}-S_{\bar{\Q}_p})/\Z/l)$ for all $r < r_o$ 
 and similary over ${\bar{\F}_p}$. Consider the diagram
 $$
 \xymatrix{
 F_l^{< r_o}(S_{\bar{\Q}_p}) \ar@{->>}[r]^{\Psi_{\bar{\Q}_p}} \ar[d]_{\cong} & A^{< r_o}_{\Z/l \wr G_{\bar{\Q}_p}} \ar[d]  \\
 F_l^{< r_o}(S_{\bar{\F}_p}) \ar@{->>}[r]^{\Psi_{\bar{\F}_p}} & A^{< r_o}_{\Z/l \wr G_{\bar{\F}_p}} 
 }
 $$
 where the left vertical arrow is induced by the specialization map for $G$. It is an isomorphism by condition (i). 
 The surjectivity of the horizontal maps follow from \cite[Lemma 8.1 (2)]{To}, whose proof is valid over an arbitrary field. We can apply \cite[Lemma 8.1 (2)]{To}
 since condition (ii) holds for $G$. The diagram commutes by Lemma \ref{LeCommutative}. Adding the cycle maps to etale cohomology we obtain a commutative diagram
 $$
 \xymatrix{
 F_l^{< r_o}(S_{\bar{\Q}_p}) \otimes \Z_l \ar@{->>}[r] \ar[d]_{\cong} & A^{< r_o}_{\Z/l \wr G_{\bar{\Q}_p}} \otimes \Z_l \ar[d] \ar[r] 
 & H^{< 2r_o}(B(\Z/l \wr G_{\bar{\Q}_p}),\Z_l) \ar[d]^{\cong} \\
 F_l^{< r_o}(S_{\bar{\F}_p}) \otimes \Z_l \ar@{->>}[r] & A^{< r_o}_{\Z/l \wr G_{\bar{\F}_p}} \otimes \Z_l  \ar[r] & H^{< 2r_o}(B(\Z/l \wr G_{\bar{\F}_p}),\Z_l).
 }
 $$
 Here the right vertical map is an isomorphism by Corollary \ref{CorSpezCoh}.
 Since condition (iii) holds for $G$, it follows from Lemma \ref{LeSplitInjective} that the composition
 $$
  \xymatrix{
 F_l^{< r_o}(S_{\bar{\Q}_p}) \otimes \Z_l \ar@{->>}[r] & A^{< r_o}_{\Z/l \wr G_{\bar{\Q}_p}} \otimes \Z_l \ar[r] & H^{< 2r_o}(B(\Z/l \wr G_{\bar{\Q}_p}),\Z_l) 
 }
 $$
 is split injective, using that the target is equal to $H^{< 2r_o}((S^l_{\bar{\Q}_p}-S_{\bar{\Q}_p})/\Z/l,\Z_l)$ by definition.
 This shows that condition (i) and (iii) hold again for $\Z/l \wr G$. The fact that condition (ii) also holds for $\Z/l \wr G$ is proven in \cite[Lemma 8.1 (2)]{To}.
 \end{proof}
 
 \section{Spezialization for some Chevalley Groups}
 
 \begin{Proposition}
\label{PropSpezCl}
\begin{equivlist}
 \item The specialization map $A^*BG_{\C} \to A^*BG_{\bar{\F}_p}$ for the classical groups $G=\GL_n(\F_q)$, $\Sp_{2m}(\F_q)$, $\Oo_n(\F_q)$ and 
 $\SO_n(\F_q)$ over the finite field $\F_q$ of characteristic $p$ become injective after inverting $2p$.
 \item If $q \equiv 1$ mod $4$ the specialization map for $\GL_n(\F_q)$, $\Sp_{2m}(\F_q)$ and $\Oo_{2m+1}(\F_q)$ become injective after inverting $p$.
 \item If $S$ denotes the product of $p$ and all prime divisor of $q-1$ the specialization map for $\SL_n(\F_q)$ becomes injective after inverting $S$.
\end{equivlist}
\end{Proposition}

\begin{proof}
We need to check that the specialization map of the respective $l$-Sylow subgroups are injective, where in (i) we consider a prime $l$ not dividing $2p$, in (ii) 
we consider a prime $l$ different from $p$ and in (iii) we consider a prime $l$ not dividing $S$. 

We claim that in any case these $l$-Sylow subgroups are products of iterated wreath products of the form $\mathbb{Z}/l^{\wr i} \wr \mathbb{Z}/l^b$.
In \cite{We} Weir proves the assertion in the case $l \neq 2,p$ for the groups $\GL_n(\F_q)$, $\Sp_{2m}(\F_q)$, $\Oo_n(\F_q)$ and $\SO_n(\F_q)$. 
Using an argument of Quillen used in\cite[Lemma 13]{Qu} one easily sees that the same  is true for the $2$-Sylow subgroups of $\GL_n(\F_q)$, $\Sp_{2m}(\F_q)$ and 
$\Oo_{2m+1}(\F_q)$ provided $q \equiv 1$ mod $4$. Finally if $l$ is a prime not dividing $S$ then every $l$-Sylow subgroup of $\SL_n(\F_q)$ is also an 
$l$-Sylow subgroup of $\GL_n(\F_q)$. 

Let us check that conditions (i)-(iii) of Proposition \ref{PropWP} hold for $G=\Z/l^b$.
Choosing an $l^b$-th root of unity we may identify $(\Z/l^b)_k = \mu_{l^b,k}$ for $k=\bar{\Q}_p, \bar{\F}_p$. Note that reduction induces an 
isomorphism $\mu_{l^b}(\bar{\Q}_p)=\mu_{l^b}(\bar{\F}_p)$.
It follows from the Kummer sequence and \cite[Lemma 2.4]{To2}] that $A^*_{\mu_n}=\Z[t]/(nt)$, where $t$ is the first chern class of the character 
$\mu_n \hookrightarrow \G_m$. In particular $A^*_{\mu_n}$ is independent of the base field. Hence (i) holds for $G=\Z/l^b$.
Condition (iii) is an immediate consequence of Lemma \ref{LeCyclemapmu}.
We know that the complement of the zero section in $\mathcal{O}_{\Pbb^m}(l^b)$ approximates $B(\mu_{l^b})$ (cf. the proof of Lemma \ref{LeCyclemapmu}). 
This space can be cut open into spaces of the form $(\mathbb{A}^1-\{0\}) \times \mathbb{A}^k$. Hence condition (ii) holds.
Then using Proposition \ref{PropWP} we see that the specialization map for the group $\mathbb{Z}/l\mathbb{Z}^{\wr i} \wr \mathbb{Z}/l^b \mathbb{Z}$ is an 
isomorphism. The general case follows from the Künneth formula (\cite[Lemma 2.12]{To2}).
\end{proof}

\begin{Remark}
A $p$-Sylow subgroup $U$ of $\GL_n(\mathbb{F}_q)$ is given by the upper triangular matrices with $1$'s on the diagonal. The specialization morphism for $U$
is not injective. For $n=2$ this Sylow subgroup is just $(\mathbb{Z}/p\mathbb{Z})^a$ with $a=v_p(q)$ and its Chow ring in characteristic $p$ is trivial 
(as follows from the Artin Schreier sequence) while in characteristic $0$ it is $\mathbb{Z}[t_1,\ldots,t_a]/(pt_1,\ldots,pt_a)$.
\end{Remark}

\begin{Remark}
It would be interesting to know if, the specialization map $\sigma_G \colon A^*BG_{\C} \to A^*BG_{\bar{\F}_p}$ is always injective or even an isomorphism after 
inverting $p$ for any finite abstract group $G$. One could also ask if $A^*BG_{\bar{\F}_p}$ is $p$-torsion free, so that roughly
speaking $A^*BG_{\bar{\F}_p}$ equals $A^*BG_{\C}$ after taking away its $p$-torsion part.
\end{Remark}

\begin{proof}(of Theorem A.)
If $G$ is one of the groups $\GL_n$, $\Sp_{2m}$ or $\SL_n$,
Corollary \ref{Cor1} shows that the Chow ring of $BG(\F_q)_k$ is generated by Chern classes of the canonical representations of $G(\F_q)_k$, where $k$ is any field
containing $\F_q$.
The theorem thus follows from Proposition \ref{PropSpezCl} and Proposition \ref{PropSpez} (i).
\end{proof}

University of Paderborn, D-33098 Paderborn 

Dennis.Brokemper@math.uni-paderborn.de

\end{document}